\documentclass[times]{nlaauth}
\usepackage{moreverb}
\usepackage{amsmath}
\usepackage{amssymb}
\usepackage{amsfonts}
\usepackage{amsthm}
\usepackage{algorithm}
\usepackage{graphicx}
\usepackage{color}
\usepackage{epsfig}
\usepackage{mathptmx}
\usepackage{latexsym}
\usepackage[dvips,colorlinks,bookmarksopen,bookmarksnumbered,citecolor=red,urlcolor=red]{hyperref}

\newtheorem{theorem}{Theorem}
\newtheorem{lemma}[]{Lemma}

\newtheorem{definition}[]{Definition}
\newtheorem{example}[]{Example}

\newcommand{\bx}{\mathbf x}
\newcommand{\by}{\mathbf y}

\newcommand{\zero}{\mathbf 0}

\newcommand\BibTeX{{\rmfamily B\kern-.05em \textsc{i\kern-.025em b}\kern-.08em
T\kern-.1667em\lower.7ex\hbox{E}\kern-.125emX}}
\def\volumeyear{2010}

\begin{document}

\title{An inexact Noda iteration for computing the smallest
eigenpair of a large irreducible monotone matrix}
\author{Ching-Sung Liu}

\begin{abstract}
In this paper, we present an inexact Noda iteration with inner-outer iterations for finding the smallest eigenvalue and the associated eigenvector of an irreducible monotone matrix. The proposed inexact Noda iteration contains two main relaxation steps for computing the smallest eigenvalue and the associated eigenvector, respectively. These relaxation steps depend on the relaxation factors, and we analyze how the relaxation factors in the relaxation steps affect the convergence of the outer iterations. By considering two different relaxation factors for solving the inner linear systems involved, we prove that the convergence is globally linear or superlinear, depending on the relaxation factor, and that the relaxation factor also influences the convergence rate. The proposed inexact Noda iterations are structure preserving and maintain the positivity of approximate eigenvectors. Numerical examples are provided to illustrate that the proposed inexact Noda iterations are practical, and they always preserve the positivity of approximate eigenvectors.
\end{abstract}

\maketitle

\runningheads{C.-S.~Liu}{INEXACT NODA ITERATION}

\address{Department of Applied Mathematics, National Chiao Tung University, Taiwan.}

\corraddr{Ching-Sung Liu, Department of Applied Mathematics, National Chiao Tung University, Taiwan.
E-mail: chingsungliu@nctu.edu.tw}


\vspace{-6pt}

\section{Introduction}

Monotone matrices arise in many areas of mathematics, such as stability
analysis \cite{OLES83}, and bounds for eigenvalues and singular values \cite%
{Axels90,Axels03}. In many applications, one is interested in finding the
smallest eigenvalue $\lambda $ and the associated eigenvector $\mathbf{x}$
of an irreducible nonsingular monotone matrix $A\in \mathbb{R}^{n\times n}$.
The smallest eigenvalue $\lambda$ of a monotone matrix $A$ is defined as $%
\sigma _{\min }(A)=\min\{| \lambda | \mid \lambda \in \sigma(A) \}$, where $%
\sigma (A)$ denotes the set of eigenvalues of $A$. In \cite{Sch78,Sivak09},
a real matrix $A$ is called monotone if and only if $A^{-1}$ is a
non-negative matrix. The irreducible nonsingular $M$-matrices are one of the most important classes of matrices for applications such as discretized PDEs, Markov chains \cite{Abate94} and electric circuits \cite%
{Shi96}, and they have been studied extensively in the literature \cite[Chapter 6]%
{BPl94}.\ It is well known that there exist some monotone matrices that are not $M$-matrices,
such as matrices that can be written as a product of $M$-matrices.

There are some differences between an $M$-matrix and a monotone matrix. For
example, an $M$-matrix can be expressed in the form $\sigma I-B$ with a
non-negative matrix $B$ and some constant $\sigma >\rho (B)$, where $\rho
(\cdot )$ denotes the spectral radius, see \cite{BPl94}. Thus, the smallest
eigenvalue $\lambda $ of an irreducible nonsingular $M$-matrix $A$ is equal
to $\sigma -\rho (B)>0$. In contrast, the smallest eigenvalue of a monotone
matrix $A$ can only be expressed as $\sigma _{\min }(A) = \rho(A^{-1})^{-1}$.
However, the smallest eigenvalue retains the same properties \cite[p.~487]{HJo85}, that is,
the largest eigenvalue of an irreducible non-negative matrix $A^{-1}$ is the
Perron root, which is simple and equal to the spectral radius of $A^{-1}$
with a positive associated eigenvector.

For the computation of the Perron vector of a non-negative matrix $B$,
many methods exist \cite{Par98,Saa92,Ste01,Ost59b,Jia11,Jia12,LLL97,BGS06,SVo96,Lee07}
but the power methods are not structure preserving and cannot
guarantee the desired positivity of approximations when the Perron vector $%
\mathbf{x}$ has very small components. Therefore, a central concern is
how to preserve strict positivity of approximations to the Perron vector.
In $1971$, Noda introduced an inverse iteration method with shifted Rayleigh quotient-like approximations \cite{Nod71}
for non-negative matrix eigenvalue problems. This iteration method is called Noda iteration (NI), and it has also been adapted to the computation of the smallest eigenvalue and the eigenvector of an irreducible nonsingular $M$-matrix \cite%
{Xue96,AXY02}. The major advantages of Noda iteration are structure
preservation and global convergence. More precisely, it generates a
monotonically decreasing sequence of approximate eigenvalues that is
guaranteed to converge to $\rho (B)$, and maintains the positivity of
approximate eigenvectors. Furthermore, the convergence has been proven to be
superlinear \cite{Nod71} and asymptotically quadratic \cite{Els76}. In \cite%
{Jia14}, the authors introduced two inexact strategies for Noda
iteration, which are called inexact Noda iteration (INI) to find the Perron
vector of a non-negative matrix (or $M$-matrix). The proposed INI
algorithms are practical, and they always preserve the positivity of
approximate eigenvectors. Moreover, the convergence of INI with these
two strategies is globally linear and superlinear with convergence order $\frac{1+\sqrt{5}}{2}$, respectively.

In this paper, we propose an inexact Noda iteration (INI) to find the
smallest eigenvalue and the associated eigenvector of an irreducible
monotone matrix $A$. The major contribution of this paper is to provide two
main relaxation steps for computing the smallest eigenvalue $\lambda$ and
the associated eigenvector $\mathbf{x}$, respectively. The first step is to
use $O(\gamma_{k} \min (\mathbf{x}_{k}))$ as a stopping criterion for inner
iterations, with $0< \gamma_{k} <1 $, where $\mathbf{x}_k$ is the current
positive approximate eigenvector. The second step is to update the
approximate eigenvalues by using the recurrence relations $\overline{\lambda
}_{k+1}=\overline{\lambda }_{k}-\left( 1-\gamma _{k}\right) \min \left(
\frac{\mathbf{x}_{k}}{\mathbf{y}_{k+1}}\right)$, where $y_{k+1}$ is the next
normalized positive approximate eigenvector, so resulting INI
algorithms are structure preserving and globally convergent. The above
parameter $\gamma_k$ is called the \textquotedblleft relaxation
factor\textquotedblright. We then establish a rigorous convergence theory of
INI with two different relaxation factors $\gamma _{k}$, and prove that
the convergence of the resulting INI algorithms is globally linear, and
superlinear with the relaxation factor $\gamma _{k}$ as the convergence
rate, respectively.

In fact, the inner iterations of INI (or NI) require the solution of ill-conditioned linear systems when the sequence of approximate eigenvalues converges to $\rho(A^{-1})$ (or $\rho(B)$). In order to reduce the condition number of the inner linear system, we propose a modified Noda iteration (MNI) by using rank one update for the inner iterations, and we show that MNI and NI are mathematically equivalent. For monotone matrix eigenvalue problems, we also develop an integrated algorithm that combines INI with MNI, and we call this modified inexact Noda iteration (MINI). This hybrid iterative method can significantly improve the condition number of inner linear systems of INI.

The paper is organized as follows. In Section 2, we introduce the Noda iteration and some preliminaries. Section 3 contains the new strategy for inexact Noda iteration, and proves some basic properties for it. In Section 4, we establish its convergence theory, and derive the asymptotic convergence factor precisely. In Section 5, we present the integrated algorithm that combines INI with MNI. Finally, in Section 6 we present some numerical examples illustrating the convergence theory and the effectiveness of INI, and we make some concluding remarks in Section 7.

\section{Preliminaries and Notation}

For any real matrix $B=\left[ b_{ij}\right] \in \mathbb{R}^{n\times n}$, we
write $B \geq 0$ $(>0)$ if $b_{ij} \geq 0$ $(>0)$ for all $1\leq i,j\leq n$.
We define $|B|=[|b_{ij}|]$. If $B \geq 0$, we say $B$ is a non-negative
matrix, and if $B > 0$, we say $B$ is a positive matrix. For real matrices $%
B $ and $C$ of the same size, if $B-C$ is a non-negative matrix, we write $%
B\geq C $. A non-negative (positive) vector is similarly defined. A
non-negative matrix $B$ is said to be reducible if it can be placed into
block upper-triangular form by simultaneous row/column permutations;
otherwise it is irreducible. If $\mu $ is not an eigenvalue of $B$, the
function $\mathrm{sep}(\mu ,B)$ is defined as
\begin{equation}
\mathrm{sep}(\mu ,B)=\Vert (\mu I-B)^{-1}\Vert ^{-1}.  \label{eq:sep}
\end{equation}
$\angle (\mathbf{w},\mathbf{z})$ denotes the acute angle of any two nonzero
vectors $\mathbf{w}$ and $\mathbf{z}$. Throughout the paper, we use a $2$%
-norm for vectors and matrices, and the superscript $T$ denotes its
transpose.

We review some fundamental properties of non-negative matrices, monotone
matrices and $M$-matrices.

\begin{definition}
A matrix $A$ is said to be \textquotedblleft monotone\textquotedblright\ if $%
Ax\geq 0$ implies $x\geq 0$ for any positive vector. 
\end{definition}

Another characterization of monotone matrices is given by the following well
known theorem.

\begin{theorem}[\protect\cite{coll60}]
\label{thm:mono} $A$ is monotone if and only if $A$ is non-singular and $%
A^{-1}\geq 0.$
\end{theorem}

\begin{definition}
A monotone matrix $M\ $is an $M$-matrix if $M=\left( m_{ij}\right) $, $%
m_{ij}\leq 0$ for $i\neq j.$
\end{definition}

\begin{lemma}[\protect\cite{BPl94}]
\label{thm:M} Let $M$ is a nonsingular $M$-matrix. Then the following
statements are equivalent:

\begin{enumerate}
\item $M= \left(a_{ij}\right)$, $a_{ij}\le 0$ for $i\neq j$, and $M^{-1}\ge
0 $;

\item $M=\sigma I-B$ with some $B\ge 0$ and $\sigma > \rho(B)$.
\end{enumerate}
\end{lemma}

For a pair of positive vectors $\mathbf{v}$ and $\mathbf{w}$, define
\begin{equation*}
\max \left(\frac{\mathbf{w}}{\mathbf{v}}\right) = \underset{i}{\max}\left(%
\frac{\mathbf{w}^{(i)}}{\mathbf{v}^{(i)}}\right),\text{ \ }\min\left(\frac{%
\mathbf{w}}{\mathbf{v}}\right) = \underset{i}{\min}\left(\frac{\mathbf{w}%
^{(i)}} {\mathbf{v}^{(i)}}\right),
\end{equation*}
where $\mathbf{v}=[\mathbf{v}^{(1)},\mathbf{v}^{(2)},\ldots,\mathbf{v}%
^{(n)}]^T$ and $\mathbf{w}=[\mathbf{w}^{(1)},\mathbf{w}^{(2)},\ldots,\mathbf{%
w}^{(n)}]^T$. The following lemma gives bounds for the spectral radius of a
non-negative matrix $B$.

\begin{lemma}[{{\protect\cite[p.~493]{HJo85}}}]
\label{maxmin}Let $B\ $be an irreducible non-negative matrix. If $\mathbf{v}>%
\mathbf{0}$ is not an eigenvector of $B$, then
\begin{equation}
\min \left( \frac{B\mathbf{v}}{\mathbf{v}}\right) <\rho (B)<\max \left(
\frac{B\mathbf{v}}{\mathbf{v}}\right) .  \label{eq:maxmin}
\end{equation}
\end{lemma}

\subsection{The Noda iteration}

The Noda iteration \cite{Nod71} is an inverse iteration shifted by a
Rayleigh quotient-like approximation of the Perron root of an irreducible
non-negative matrix $B$.

Given an initial vector $\mathbf{x}_{0}>\mathbf{0}$ with $\Vert \mathbf{x}%
_{0}\Vert =1$, the Noda iteration (NI) consists of three steps:
\begin{align}
(\widehat{\lambda }_{k}I-B)\,\mathbf{y}_{k+1}& =\mathbf{x}_{k},
\label{eq:step1} \\
\mathbf{x}_{k+1}& =\mathbf{y}_{k+1}\,/\,\Vert \mathbf{y}_{k+1}\Vert ,
\label{eq:step2} \\
\widehat{\lambda }_{k+1}& =\max \left( \frac{B\mathbf{x}_{k+1}}{\mathbf{x}%
_{k+1}}\right) .  \label{eq:step3}
\end{align}

The main task is to compute a new approximation $\mathbf{x}_{k+1}$ to $%
\mathbf{x}$ by solving the inner linear system (\ref{eq:step1}). From Lemma~%
\ref{maxmin}, we know that $\widehat{\lambda }_{k}>$ $\rho (B)$ if $\mathbf{x%
}_{k}$ is not a scalar multiple of the eigenvector $\mathbf{x}$.
This result shows that $\widehat{\lambda }_{k}I-B$ is an irreducible nonsingular $M$-matrix,
and its inverse is an irreducible non-negative matrix. Therefore,
we have $\mathbf{y}_{k+1}>\mathbf{0}$ and $\mathbf{x}_{k+1}>\mathbf{0}$, i.e., $\mathbf{x}_{k+1}$
is always a positive vector if $\mathbf{x}_{k}$ is. After variable transformation,
we get $\widehat{\lambda }_{k+1}$ from the following relation:
\begin{equation*}
\widehat{\lambda }_{k+1}=\widehat{\lambda }_{k}-\min \left( \frac{\mathbf{x}%
_{k}}{\mathbf{y}_{k+1}}\right) ,
\end{equation*}%
so $\{\widehat{\lambda }_{k}\}$ is monotonically decreasing.

\subsection{The inexact Noda iteration}

\label{sec:INVITTHM} The inexact Noda iteration. Based on the Noda iteration,
in \cite{Jia14} the authors propose an inexact Noda iteration (INI) for the computation of the spectral radius of a non-negative irreducible matrix $B$. In this paper, since $A$ is a monotone matrix, i.e.,
$A^{-1}$ is a non-negative matrix, we replace $B$ by $A^{-1}$ in (\ref{eq:step1}), i.e.,
\begin{equation}
(\widehat{\lambda }_{k}I-A^{-1})\,\mathbf{y}_{k+1}= \mathbf{x}_{k}.
\label{eq: exrelation}
\end{equation}
When $A$ is large and sparse, we see that we must resort to an iterative linear solver to get an approximate solution. In order to reduce the computational cost of (\ref{eq: exrelation}), we solve $\mathbf{y}_{k+1}$ in (\ref{eq: exrelation}) by inexactly satisfying
\begin{equation}
(\widehat{\lambda }_{k}I-A^{-1})\,\mathbf{y}_{k+1}=\mathbf{x}_{k}+A^{-1}%
\mathbf{f}_{k},  \label{eq: ine relation}
\end{equation}%
which is equivalent to
\begin{align}
(\widehat{\lambda }_{k}A-I)\,\mathbf{y}_{k+1}& =A\mathbf{x}_{k}+\mathbf{f}%
_{k},\text{{}}  \label{eq:inexactsys} \\
\mathbf{x}_{k+1}& =\mathbf{y}_{k+1}/\Vert \mathbf{y}_{k+1}\Vert ,
\label{eq:normal}
\end{align}%
where $\mathbf{f}_{k}$ is the residual vector between $(\widehat{\lambda }%
_{k}A-I)\,\mathbf{y}_{k+1}$ and A$\mathbf{x}_{k}$. Here, the residual norm
(inner tolerance) $\xi _{k}:=\Vert \mathbf{f}_{k}\Vert $ can be changed at
each iterative step $k$.\noindent

%

\begin{theorem}[\protect\cite{Jia14}]
\label{monotone}Let $A$ be an irreducible monotone matrix and $0\leq \gamma
<1$ be a fixed constant. For the unit length $\mathbf{x}_{k}>\mathbf{0}$, if
$\mathbf{x}_{k}\not=\mathbf{x}$ and $\mathbf{f}_{k}$ in (\ref{eq:inexactsys}%
) satisfies
\begin{equation}
\left\Vert A^{-1}\mathbf{f}_{k}\right\Vert \leq \gamma \min \left( \,\mathbf{%
x}_{k}\right) ,  \label{inexact condi}
\end{equation}%
then $\{\widehat{\lambda }_{k}\}$ is monotonically decreasing and $%
\lim_{k\rightarrow \infty }\widehat{\lambda }_{k}=\rho (A^{-1})$. Moreover,
the convergence of INI is at least globally linear.
\end{theorem}

\noindent Based on (\ref{eq:inexactsys})-(\ref{inexact condi}), we describe INI as Algorithm~\ref{alg:iiva}.

\begin{algorithm}
\begin{enumerate}
  \item   Given $\widehat{\lambda}_{0}$, $\bx_0 > \zero$ with $\Vert \bx_0\Vert =1$, $0\le \gamma < 1$ and ${\sf tol}>0$.
  \item   {\bf for} $k =0,1,2,\dots$
  \item   \quad Solve $(\widehat{\lambda}_{k}A-I)\,\mathbf{y}_{k+1}=A\mathbf{x}_{k}$
  inexactly such that the inner tolerance $\xi_{k}$ satisfies condition (\ref{inexact condi})
  \item   \quad Normalize the vector $\bx_{k+1}= \by_{k+1}/\Vert \by_{k+1}\Vert$.
  \item   \quad Compute $\widehat{\lambda}_{k+1}= \max \left(\frac{A^{-1}\bx_{k+1}}{\bx_{k+1}}\right)$.
  \item   {\bf until} convergence: Resi$= \Vert A\bx_{k+1}-\widehat{\lambda}_{k}^{-1}\bx_{k+1}\Vert <{\sf tol}$.
\end{enumerate}
\caption{Inexact Noda Iteration (INI)}
\label{alg:iiva}
\end{algorithm}

Using the relation (\ref{eq: ine relation}), step 5 in Algorithm~\ref%
{alg:iiva} can be rewritten as
\begin{equation*}
\widehat{\lambda }_{k+1}=\widehat{\lambda }_{k}-\min \left( \frac{\mathbf{x}%
_{k}+A^{-1}\mathbf{f}_k} {\mathbf{y}_{k+1}}\right).
\end{equation*}%
Unfortunately, $A^{-1}$ is not explicitly available; in other words,
we need to compute \textquotedblleft $A^{-1}\mathbf{f}_k$\textquotedblright exactly for the required approximate eigenvalue $\widehat{\lambda }_{k+1}$. Hence, in the next section, we propose a new strategy to estimate the approximate eigenvalues without increasing the computational cost. This strategy is practical and preserves the strictly decreasing property of the approximate eigenvalue sequence.

\section{The relaxation strategy for INI and some basic properties}

\label{sec:type1}

In order to ensure that INI is correctly implemented, we now propose two main relaxation steps to define Algorithm~\ref{alg:ini}:

\begin{itemize}
\item The residual norm satisfies
\begin{equation}
\xi _{k}=\Vert \mathbf{f}_{k}\Vert \leq \gamma _{k}\mathrm{sep}(0,A)\min (%
\mathbf{x}_{k}),  \label{eq:strategy}
\end{equation}%
where $0\leq \gamma _{k}\leq \gamma <1$ with a constant upper bound $\gamma
. $
\end{itemize}

\begin{itemize}
\item The update of the approximate eigenvalue satisfies%
\begin{equation}
\overline{\lambda }_{k+1}=\overline{\lambda }_{k}-\left( 1-\gamma
_{k}\right) \min \left( \frac{\mathbf{x}_{k}}{\mathbf{y}_{k+1}}\right) .
\label{eq:updateeig}
\end{equation}
\end{itemize}

\begin{algorithm}
\begin{enumerate}
  \item   Given $\overline{\lambda }_{0}$, $\bx_0 > \zero$ with $\Vert \bx_0\Vert =1$, $0\le \gamma < 1$ and ${\sf tol}>0$.
  \item   {\bf for} $k =0,1,2,\dots$
  \item   \quad Solve $(\overline{\lambda }_{k}A-I)\,\mathbf{y}_{k+1}=A\mathbf{x}_{k}$
  inexactly such that the inner tolerance $\xi_{k}$ satisfies condition (\ref{eq:strategy})
  \item   \quad Normalize the vector $\bx_{k+1}= \by_{k+1}/\Vert \by_{k+1}\Vert$.
  \item   \quad Compute $\overline{\lambda }_{k+1}$  that satisfies condition (\ref{eq:updateeig}).
  \item   {\bf until} convergence: $\Vert A\bx_{k+1}-\overline{\lambda }_{k}^{-1}\bx_{k+1}\Vert <{\sf tol}$.
\end{enumerate}
\caption{Inexact Noda Iteration for monotone matrices (INI)}
\label{alg:ini}
\end{algorithm}

In step 3 of Algorithm~\ref{alg:ini}, it leads to two equivalent inexact
relation satisfying
\begin{align}
(\overline{\lambda }_{k}I-A^{-1})\,\mathbf{y}_{k+1}& =\mathbf{x}_{k}+A^{-1}%
\mathbf{f}_{k}  \label{eq: ine relation2} \\
(\overline{\lambda }_{k}A-I)\,\mathbf{y}_{k+1}& =A\mathbf{x}_{k}+\mathbf{f}%
_{k},\text{{}}  \label{eq:inexactsys2}
\end{align}%
We remark that $\overline{\lambda }_{k+1}$ in (\ref{eq:updateeig}) is no
longer equal to $\max \left( \frac{A^{-1}\mathbf{x}_{k+1}}{\mathbf{x}_{k+1}}%
\right)$, and therefore that $\overline{\lambda }_{k+1}$ cannot be clearly
demonstrated to be greater than its lower bound $\rho (A^{-1})$. The following lemma ensures
that $\rho (A^{-1})$ is still the lower bound of $\overline{\lambda}_k$.

\begin{lemma}
\label{monotone2}Let $A$ be an irreducible monotone matrix. For the unit
length $\mathbf{x}_{k}\not=\mathbf{x}>\mathbf{0}$ and the relaxation factor $%
\gamma_k \in [0,1)$, if $\overline{\lambda }_{k}>\rho (A^{-1}) $, $\mathbf{f}%
_{k}$ in (\ref{eq:inexactsys2}) satisfies condition (\ref{eq:strategy}) and
the approximate eigenvalue satisfies (\ref{eq:updateeig}), then the new
approximation $\mathbf{x}_{k+1}>\mathbf{0}$ and the sequence $\left\{
\overline{\lambda }_{k}\right\} $ is monotonically decreasing and bounded
below by $\rho (A^{-1})$, i.e.,%
\begin{equation*}
\overline{\lambda }_{k}>\overline{\lambda }_{k+1}\geq \rho (A^{-1}).
\end{equation*}
\end{lemma}

\begin{proof}
From (\ref{eq:updateeig}) and $\gamma _{k}\in [0,1) $, it is easy to know
that $\left\{ \overline{\lambda }_{k}\right\} $ is monotonically decreasing,
i.e., $\overline{\lambda }_{k}>\overline{\lambda }_{k+1}.$ \newline
From (\ref{eq:strategy}),
\begin{equation}
\left\Vert A^{-1}\mathbf{f}_{k}\right\Vert \leq \left\Vert A^{-1}\right\Vert
\left\Vert \mathbf{f}_{k}\right\Vert \leq \gamma _{k}\min (\mathbf{x}_{k}),
\label{eq:normAf}
\end{equation}%
which implies $\left\vert A^{-1}\mathbf{f}_{k}\right\vert \leq \gamma _{k}%
\mathbf{x}_{k},$ then $\mathbf{x}_{k}+A^{-1}\mathbf{f}_{k}>0.$ Consequently,
$\overline{\lambda }_{k}I-A^{-1}$ is a nonsingular $M$-matrix, and the
vector $\mathbf{y}_{k+1}$ satisfies
\begin{equation*}
\mathbf{y}_{k+1}=(\overline{\lambda }_{k}I-A^{-1})^{-1}\left( \mathbf{x}%
_{k}+A^{-1}\mathbf{f}_{k}\right) >\mathbf{0}.
\end{equation*}%
This implies $\mathbf{x}_{k+1}=\mathbf{y}_{k+1}\,/\,\Vert \mathbf{y}%
_{k+1}\Vert >\mathbf{0}$.

We now prove $\overline{\lambda }_{k}$ is bounded below by $\rho (A^{-1}).$
From (\ref{eq:normAf}), we have
\begin{equation}
\left( 1-\gamma _{k}\right) \mathbf{x}_{k}\leq \mathbf{x}_{k}+A^{-1}\mathbf{f%
}_{k}\leq \left( 1+\gamma _{k}\right) \mathbf{x}_{k},  \label{eq:inqu_gamma}
\end{equation}%
and
\begin{equation*}
\left( 1-\gamma _{k}\right) \frac{\mathbf{x}_{k}}{\mathbf{y}_{k+1}}\leq
\frac{\mathbf{x}_{k}+A^{-1}\mathbf{f}_{k}}{\mathbf{y}_{k+1}}\leq \left(
1+\gamma _{k}\right) \frac{\mathbf{x}_{k}}{\mathbf{y}_{k+1}}.
\end{equation*}%
Hence, we obtain%
\begin{equation*}
\left( 1-\gamma _{k}\right) \min \left( \frac{\mathbf{x}_{k}}{\mathbf{y}%
_{k+1}}\right) \leq \min \left( \frac{\mathbf{x}_{k}+A^{-1}\mathbf{f}_{k}}{%
\mathbf{y}_{k+1}}\right) \leq \left( 1+\gamma _{k}\right) \min \left( \frac{%
\mathbf{x}_{k}}{\mathbf{y}_{k+1}}\right) ,
\end{equation*}%
then
\begin{equation}
\left\vert \min \left( \frac{\mathbf{x}_{k}+A^{-1}\mathbf{f}_{k}}{\mathbf{y}%
_{k+1}}\right) -\min \left( \frac{\mathbf{x}_{k}}{\mathbf{y}_{k+1}}\right)
\right\vert \leq \gamma _{k}\min \left( \frac{\mathbf{x}_{k}}{\mathbf{y}%
_{k+1}}\right) .  \label{eq:gap}
\end{equation}%
From (\ref{eq: ine relation2}), it follows that%
\begin{equation}
\text{ }\rho (A^{-1})\leq \widehat{\lambda }_{k+1}=\max \left( \frac{A^{-1}%
\mathbf{x}_{k+1}}{\mathbf{x}_{k+1}}\right) =\overline{\lambda }_{k}-\min
\left( \frac{\mathbf{x}_{k}+A^{-1}\mathbf{f}_{k}}{\mathbf{y}_{k+1}}\right) .
\label{eq:updatelam}
\end{equation}%
Combine (\ref{eq:updateeig}), (\ref{eq:updatelam}) and (\ref{eq:gap}), then
\begin{eqnarray}
\overline{\lambda }_{k+1} &=&\overline{\lambda }_{k}-\left( 1-\gamma
_{k}\right) \min \left( \frac{\mathbf{x}_{k}}{\mathbf{y}_{k+1}}\right)
\notag \\
&=&\widehat{\lambda }_{k+1}+\min \left( \frac{\mathbf{x}_{k}+A^{-1}\mathbf{f}%
_{k}}{\mathbf{y}_{k+1}}\right) -\left( 1-\gamma _{k}\right) \min \left(
\frac{\mathbf{x}_{k}}{\mathbf{y}_{k+1}}\right)  \label{eq: lam_relation} \\
&\geq &\widehat{\lambda }_{k+1}+\left( 1-\gamma _{k}-\left( 1-\gamma
_{k}\right) \right) \min \left( \frac{\mathbf{x}_{k}}{\mathbf{y}_{k+1}}%
\right)  \notag \\
&>&\rho (A^{-1}).  \notag
\end{eqnarray}%
By induction, $\left\{ \overline{\lambda }_{k}\right\} $ is bounded below by
$\rho (A^{-1}),$ i.e.,
\begin{equation*}
\overline{\lambda }_{k}>\rho (A^{-1})\text{ for all }k.
\end{equation*}
\end{proof}

From Lemma~\ref{monotone2}, since $\left\{ \overline{\lambda }_{k}\right\} $
is a monotonically decreasing and bounded sequence, we must have $%
\lim_{k\rightarrow \infty }\overline{\lambda }_{k}=\alpha \geq \rho (A^{-1})$%
, where $\alpha =\rho (A^{-1})$ or $\alpha >\rho (A^{-1})$. We next
investigate the case $\alpha >\rho (A^{-1})$, and present some basic results; this plays an important role later in proving the convergence of INI.

\begin{lemma}
\label{opposite}For Algorithm~\ref{alg:ini}, if $\overline{\lambda }_{k}$ is
converge to $\alpha >\rho (A^{-1})$, then (i) $\Vert \mathbf{y}_{k}\Vert $is
bounded$;$(ii) $\underset{k\rightarrow \infty }{\lim }\min (\mathbf{x}%
_{k})=0;$(iii) $\sin \angle \left( \mathbf{x}_{k},\mathbf{x}\right) \geq m
>0 $ for some constant $m >0$, where $\angle (\mathbf{x_k},\mathbf{x})$ the
acute angle of $\mathbf{x_k}$ and $\mathbf{x}$.
\end{lemma}

\begin{proof}
(i). From (\ref{eq:inqu_gamma}), we get
\begin{align}
\Vert \mathbf{y}_{k+1}\Vert & =\Vert \left( \overline{\lambda }%
_{k}I-A^{-1}\right) ^{-1}\left( \mathbf{x}_{k}+A^{-1}\mathbf{f}_{k}\right)
\Vert \leq (1+\gamma _{k})\Vert (\overline{\lambda }_{k}I-A^{-1})^{-1}\Vert
\notag \\
& =2\mathrm{sep}(\overline{\lambda }_{k},A^{-1})^{-1}\leq 2\mathrm{sep}%
(\alpha ,A^{-1})^{-1}<\infty .  \label{eq:bdyk}
\end{align}%
(ii). From (\ref{eq:updateeig}) it follows that
\begin{equation}
\underset{k\rightarrow \infty }{\lim }\min \left( \frac{\mathbf{x}_{k}}{%
\mathbf{y}_{k+1}}\right) =\underset{k\rightarrow \infty }{\lim }\left( \frac{%
\overline{\lambda }_{k}-\overline{\lambda }_{k+1}}{1-\gamma _{k}}\right) =0.
\label{eq: minfk}
\end{equation}%
On the other hand, from (\ref{eq:bdyk}) and (\ref{eq: minfk}) we have%
\begin{equation*}
\min \left( \frac{\mathbf{x}_{k}}{\mathbf{y}_{k+1}}\right) \geq \frac{\min (%
\mathbf{x}_{k})}{\max \left( \mathbf{y}_{k+1}\right) }\geq \frac{\min (%
\mathbf{x}_{k})\mathrm{sep}(\alpha ,A^{-1})}{2}>0.
\end{equation*}%
Thus, it is holds that%
\begin{equation}
\underset{k\rightarrow \infty }{\lim }\min (\mathbf{x}_{k})=0.
\label{eq:minxkfk}
\end{equation}%
(iii) Suppose there is a subsequence $\{\sin \angle (\mathbf{x}_{k_{j}},%
\mathbf{x})\}$ which converges to zero, then
\begin{equation*}
\underset{j\rightarrow \infty }{\lim }\widehat{\lambda }_{k_{j}}=\underset{%
j\rightarrow \infty }{\lim }\max \left( \frac{A^{-1}\mathbf{x}_{k_{j}}}{%
\mathbf{x}_{k_{j}}}\right) =\max \left( \underset{j\rightarrow \infty }{\lim
}\frac{A^{-1}\mathbf{x}_{k_{j}}}{\mathbf{x}_{k_{j}}}\right) =\rho (A^{-1}).
\end{equation*}%
By the definition of $\widehat{\lambda }_{k},$ from (\ref{eq: lam_relation})
and (\ref{eq:gap}), we have%
\begin{eqnarray}
\left\vert \widehat{\lambda }_{k}-\overline{\lambda }_{k}\right\vert
&=&\left\vert \min \left( \frac{\mathbf{x}_{k-1}+A^{-1}\mathbf{f}_{k-1}}{%
\mathbf{y}_{k}}\right) -\left( 1-\gamma _{k}\right) \min \left( \frac{%
\mathbf{x}_{k-1}}{\mathbf{y}_{k}}\right) \right\vert  \notag \\
&\leq &\left\vert \left( 1+\gamma _{k}-\left( 1-\gamma _{k}\right) \right)
\min \left( \frac{\mathbf{x}_{k-1}}{\mathbf{y}_{k}}\right) \right\vert \text{
for any }k.  \label{eq: boundlambdagap}
\end{eqnarray}%
From (\ref{eq: minfk}) and (\ref{eq: boundlambdagap}),
\begin{equation*}
\underset{j\rightarrow \infty }{\lim }\overline{\lambda }_{k_{j}}=\underset{%
j\rightarrow \infty }{\lim }\left( \overline{\lambda }_{k_{j}}-\widehat{%
\lambda }_{k_{j}}+\widehat{\lambda }_{k_{j}}\right) =\rho (A^{-1}).
\end{equation*}%
This is a contradiction.
\end{proof}

\begin{lemma}[\protect\cite{Jia14}]
\label{coslower}Let $\mathbf{x}>\mathbf{0}$ be the unit length eigenvector
of $A$ associated with $\sigma _{\min }(A)$. For any vector $\mathbf{z}>%
\mathbf{0}$ with $\Vert \mathbf{z}\Vert =1$, it holds that $\cos \angle (%
\mathbf{z},\mathbf{x})>\min (\mathbf{x})$ and
\begin{equation}
\underset{\Vert \mathbf{z}\Vert =1,\,\mathbf{z}>\mathbf{0}}{\inf }\cos
\angle (\mathbf{z},\mathbf{x})=\min (\mathbf{x}).  \label{eq:coslower}
\end{equation}
\end{lemma}


\section{Convergence Analysis for INI}

\label{sec:conv}

In Sections~4.1--4.2, we will prove the global convergence and the
convergence rate of INI. Furthermore, we will derive the explicit linear
convergence factor and the superlinear convergence order with different $\gamma
_{k}$.

\subsection{Convergence Analysis}

For an irreducible non-negative matrix $A^{-1}$, recall that the largest
eigenvalue $\rho (A^{-1})$ of $A^{-1}$ is simple. Let $\mathbf{x}$ be the
unit length positive eigenvector corresponding to $\rho (A^{-1})$. Then for
any orthogonal matrix $\left[
\begin{array}{cc}
\mathbf{x} & V%
\end{array}%
\right] $ it holds (cf. \cite{GLo96}) that
\begin{equation}
\left[
\begin{array}{c}
\mathbf{x}^{T} \\
V^{T}%
\end{array}%
\right] A^{-1}\left[
\begin{array}{cc}
\mathbf{x} & V%
\end{array}%
\right] =\left[
\begin{array}{cc}
\rho (A^{-1}) & \mathbf{c}^{T} \\
0 & L%
\end{array}%
\right]  \label{eqn: partition}
\end{equation}%
with $L=V^{T}A^{-1}V$ whose eigenvalues constitute the other eigenvalues of $%
A^{-1}$.\ Therefore, we now define
\begin{equation}
\varepsilon _{k}=\overline{\lambda }_{k}-\rho (A^{-1}),\quad A_{k}=\overline{%
\lambda }_{k}I-A^{-1}.  \label{epsilonk}
\end{equation}%
Similar to (\ref{eqn: partition}) we also have the spectral decomposition
\begin{equation}
\left[
\begin{array}{c}
\mathbf{x}^{T} \\
V^{T}%
\end{array}%
\right] A_{k}\left[
\begin{array}{cc}
\mathbf{x} & V%
\end{array}%
\right] =\left[
\begin{array}{cc}
\varepsilon _{k} & 0 \\
0 & L_{k}%
\end{array}%
\right] ,  \label{eqLpart2}
\end{equation}%
where $L_{k}=\overline{\lambda }_{k}I-L$. For $\overline{\lambda }%
_{k}\not=\rho (A^{-1})$, it is easy to verify that
\begin{equation}
\left[
\begin{array}{c}
\mathbf{x}^{T} \\
V^{T}%
\end{array}%
\right] A_{k}^{-1}\left[
\begin{array}{cc}
\mathbf{x} & V%
\end{array}%
\right] =\left[
\begin{array}{cc}
\varepsilon _{k}^{-1} & \mathbf{b}_{k}^{T} \\
0 & L_{k}^{-1}%
\end{array}%
\right]
\mbox{\ \ with
$\mathbf{b}_k^T=-\frac{\mathbf{c}^TL_k^{-1}}{\varepsilon_{k}}$},
\label{eqLpart3}
\end{equation}%
from which we get
\begin{eqnarray}
A_{k}^{-1}V &=&\mathbf{x}\mathbf{b}_{k}^{T}+VL_{k}^{-1}=-\mathbf{x}\frac{%
\mathbf{c}^{T}L_{k}^{-1}}{\varepsilon _{k}}+VL_{k}^{-1},  \label{BkV} \\
A_{k}^{-1} &=&\varepsilon _{k}^{-1}\mathbf{x}\mathbf{x}^{T}-\varepsilon
_{k}^{-1}\mathbf{x}\mathbf{c}^{T}L_{k}^{-1}V^{T}+VL_{k}^{-1}V^{T},
\label{Bkinv}
\end{eqnarray}%
and
\begin{equation*}
\mathbf{x}^{T}A_{k}^{-1}=\varepsilon _{k}^{-1}\mathbf{x}^{T}-\varepsilon
_{k}^{-1}\mathbf{c}^{T}L_{k}^{-1}V^{T}.
\end{equation*}

Let $\{\mathbf{x}_{k}\}$ be generated by Algorithm~\ref{alg:ini}. We
decompose $\mathbf{x}_{k}$ into the orthogonal direct sum by%
\begin{equation}
\mathbf{x}_{k}=\mathbf{x}\cos (\varphi _{k})+\mathbf{p}_{k}\sin (\varphi
_{k}),\quad \mathbf{p}_{k}\in \text{span}(V)\perp \mathbf{x}
\label{eq:decomposition}
\end{equation}%
with $\Vert \mathbf{p}_{k}\Vert =1$ and $\varphi _{k}$ $=$ $\angle \left(
\mathbf{x}_{k},\mathbf{x}\right) $ being the acute angle between $\mathbf{x}%
_{k}$ and $\mathbf{x}$. So by definition, we have $\cos \varphi _{k}=\mathbf{%
x}^{T}\mathbf{x}_{k}$ and $\sin \varphi _{k}=\Vert V^{T}\mathbf{x}_{k}\Vert
. $ Evidently, $\mathbf{x}_{k}\rightarrow \mathbf{x}$ if and only if $\tan
\varphi _{k}\rightarrow 0$, i.e., $\sin \varphi _{k}\rightarrow 0$.

Since $\xi _{k}=\Vert \mathbf{f}_{k}\Vert \leq \gamma _{k}\mathrm{sep}%
(0,A)\min (\mathbf{x}_{k})$ in INI, it holds that $\left\vert A^{-1}\mathbf{%
f}_{k}\right\vert \leq \gamma _{k}\mathbf{x}_{k}$. Therefore, we have
\begin{equation}
\left( 1-\gamma _{k}\right) \mathbf{x}_{k}\leq \mathbf{x}_{k}+A^{-1}\mathbf{f%
}_{k}\leq \left( 1+\gamma _{k}\right) \mathbf{x}_{k}.  \label{eq: xAf_k}
\end{equation}%
As $A_{k}^{-1}\geq 0$, it follows from the above relation that
\begin{equation}
\left( 1-\gamma _{k}\right) A_{k}^{-1}\mathbf{x}_{k}\leq \mathbf{y}%
_{k+1}\leq \left( 1+\gamma _{k}\right) A_{k}^{-1}\mathbf{x}_{k}.
\label{eq: yk+1bound}
\end{equation}%
Using the above relation, we obtain

\begin{eqnarray}
\tan \varphi _{k+1} &=&\frac{\sin \varphi _{k+1}}{\cos \varphi _{k+1}}=\frac{%
\Vert V^{T}\mathbf{x}_{k+1}\Vert }{\mathbf{x}^{T}\mathbf{x}_{k+1}}=\frac{%
\Vert V^{T}\mathbf{y}_{k+1}\Vert }{\mathbf{x}^{T}\mathbf{y}_{k+1}}  \notag \\
&=&\frac{\Vert V^{T}A_{k}^{-1}(\mathbf{x}_{k}+A^{-1}\mathbf{f}_{k})\Vert }{%
\mathbf{x}^{T}A_{k}^{-1}\left( \mathbf{x}_{k}+A^{-1}\mathbf{f}_{k}\right) }
\notag \\
&=&\frac{\Vert L_{k}^{-1}V^{T}\left( \mathbf{x}_{k}+A^{-1}\mathbf{f}%
_{k}\right) \Vert }{\left( \varepsilon _{k}^{-1}\mathbf{x}^{T}-\varepsilon
_{k}^{-1}\mathbf{c}^{T}L_{k}^{-1}V^{T}\right) \left( \mathbf{x}_{k}+A^{-1}%
\mathbf{f}_{k}\right) }  \notag \\
&=&\frac{\Vert L_{k}^{-1}V^{T}\left( \mathbf{x}_{k}+A^{-1}\mathbf{f}%
_{k}\right) \Vert }{\varepsilon _{k}^{-1}\mathbf{x}^{T}\mathbf{x}%
_{k}-\varepsilon _{k}^{-1}\mathbf{c}^{T}L_{k}^{-1}V^{T}\mathbf{x}%
_{k}+\varepsilon _{k}^{-1}\mathbf{x}^{T}A^{-1}\mathbf{f}_{k}-\varepsilon
_{k}^{-1}\mathbf{c}^{T}L_{k}^{-1}V^{T}A^{-1}\mathbf{f}_{k}}  \notag \\
&\leq &\,\frac{\Vert L_{k}^{-1}\Vert \,\varepsilon _{k}\left( \sin \varphi
_{k}+\Vert A^{-1}\mathbf{f}_{k}\Vert \right) }{\cos \varphi _{k}-\mathbf{c}%
^{T}L_{k}^{-1}V^{T}\mathbf{x}_{k}-\Vert A^{-1}\mathbf{f}_{k}\Vert -\Vert
\mathbf{c}\Vert \Vert L_{k}^{-1}\Vert \Vert A^{-1}\mathbf{f}_{k}\Vert }.
\label{eq:tan}
\end{eqnarray}

Note that if we solve the inner linear system exactly, i.e., $\mathbf{f}_{k}$
$=0$, we recover NI and get
\begin{equation}
\tan \varphi _{k+1}\leq \frac{\Vert L_{k}^{-1}\Vert \,\varepsilon _{k}}{1-%
\mathbf{c}^{T}L_{k}^{-1}V^{T}\mathbf{x}_{k}/\cos \varphi _{k}}\tan \varphi
_{k}:=\beta _{k}\tan \varphi _{k}.  \label{niconv}
\end{equation}%
Since L. Elsner \cite{Els76} proved the quadratic convergence of the
proposed Noda iteration, for $k$ large enough we must have
\begin{equation*}
\beta _{k}=O(\tan \varphi _{k})\rightarrow 0.
\end{equation*}%
It follows that
\begin{equation}
\beta _{k}< \beta <1  \label{betak}
\end{equation}
for any given positive constant $\beta <1$. Therefore, we have
\begin{equation*}
\tan \varphi _{k+1}<\beta \tan \varphi _{k}
\end{equation*}
for $k\geq N$ with $N$ large enough.

\begin{theorem}[Main Theorem]
\label{main} Let $A$ be an irreducible monotone matrix. If the sequence $\{%
\overline{\lambda }_{k}\}$ is generated by INI with the relaxation
strategies (\ref{eq:strategy}) and (\ref{eq:updateeig}), then $\{\overline{%
\lambda }_{k}\}$ is monotonically decreasing and $\lim_{k\rightarrow \infty }%
\overline{\lambda }_{k}^{-1}=\sigma _{\min }(A)$.
\end{theorem}

\begin{proof}
From Lemma~\ref{monotone2}, the sequence $\left\{ \overline{\lambda }%
_{k}\right\} $ is bounded and monotonically decreasing, and we must have
either $\lim_{k\rightarrow \infty }\overline{\lambda }_{k}=\rho (A^{-1})$ or
$\lim_{k\rightarrow \infty }\overline{\lambda }_{k}=\alpha >\rho (A^{-1})$.
Next we prove by contradiction that, for INI, $\lim_{k\rightarrow \infty }%
\overline{\lambda }_{k}=\rho (A^{-1})$ must hold.

Suppose that $\lim_{k\rightarrow \infty }\overline{\lambda }_{k}=\alpha
>\rho (A^{-1})$. It follows (iii) of Lemma~\ref{opposite} show that
\begin{equation}
\frac{1}{\cos \varphi _{k}}\leq \frac{1}{\cos \varphi _{k}}\frac{\sin
\varphi _{k}}{m}=\frac{\tan \varphi _{k}}{m}.  \label{eq:lowersin}
\end{equation}%
From (ii) of Lemma~\ref{opposite}, we have
\begin{equation*}
\underset{k\rightarrow \infty }{\lim }\min (\mathbf{x}_{k})=0.
\end{equation*}%
This implies the inner tolerance $\| f_k \| \rightarrow 0$, i.e., $\Vert
A^{-1}\mathbf{f}_{k}\Vert $ is suitably small. In addition, by Lemma~\ref%
{coslower}, it holds that $\cos \varphi _{k}$ is uniformly bounded below by $%
\min (\mathbf{x})$, therefore,
\begin{equation}
(1+\Vert c\Vert \Vert L_{k}^{-1})\Vert A^{-1}\mathbf{f}_{k}\Vert /\cos
\varphi _{k}<1-\mathbf{c}^{T}L_{k}^{-1}V^{T}\mathbf{x}_{k}/\cos \varphi _{k}
\label{eq:tan2}
\end{equation}
for $k$ large enough.

Using (\ref{eq:tan}), (\ref{eq:lowersin}) and (\ref{eq:tan2}), we obtain%
\begin{align*}
\tan \varphi _{k+1}& \leq \frac{\Vert L_{k}^{-1}\Vert \,\varepsilon
_{k}\,\left( \tan \varphi _{k}+\Vert A^{-1}\mathbf{f}_{k}\Vert /\cos \varphi
_{k}\right) }{1-\mathbf{c}^{T}L_{k}^{-1}V^{T}\mathbf{x}_{k}/\cos \varphi
_{k}-(1+\Vert \mathbf{c}\Vert \Vert L_{k}^{-1}\Vert )\Vert A^{-1}\mathbf{f}%
_{k}\Vert /\cos \varphi _{k}} \\
& \leq \,\frac{\Vert L_{k}^{-1}\Vert \,\varepsilon _{k}\left( \tan \varphi
_{k}+\gamma _{k}\min (\mathbf{x}_{k})\tan \varphi _{k}/m\right) }{1-\mathbf{c%
}^{T}L_{k}^{-1}V^{T}\mathbf{x}_{k}/\cos \varphi _{k}-(1+\Vert \mathbf{c}%
\Vert \Vert L_{k}^{-1}\Vert )\gamma _{k}\min (\mathbf{x}_{k})\tan \varphi
_{k}/m} \\
& \leq \frac{\Vert L_{k}^{-1}\Vert \varepsilon _{k}\left( 1+\gamma _{k}\min (%
\mathbf{x}_{k})/m\right) }{\left( 1-\mathbf{c}^{T}L_{k}^{-1}V^{T}\mathbf{x}%
_{k}/\cos \varphi _{k}\right) -(1+\Vert \mathbf{c}\Vert \Vert
L_{k}^{-1}\Vert )\gamma _{k}\min (\mathbf{x}_{k})\tan \varphi _{k}/m}\tan
\varphi _{k}.
\end{align*}%
%
%
%
%
%
%
%
%
%
Define
\begin{equation*}
\beta _{k}^{\prime }=\frac{\Vert L_{k}^{-1}\Vert \varepsilon _{k}\left(
1+\gamma _{k}\min (\mathbf{x}_{k})/m\right) }{\left( 1-\mathbf{c}%
^{T}L_{k}^{-1}V^{T}\mathbf{x}_{k}/\cos \varphi _{k}\right) -(1+\Vert \mathbf{%
c}\Vert \Vert L_{k}^{-1}\Vert )\gamma _{k}\min (\mathbf{x}_{k})\tan \varphi
_{k}/m}.
\end{equation*}%
Note that $\beta _{k}^{\prime }$ is a continuous function with respect to $%
\min (\mathbf{x}_{k})$ for $0<\gamma _{k} <1$. Then it holds that $\beta
_{k}^{\prime }\rightarrow \beta _{k}$ defined by (\ref{niconv}) as $\min (%
\mathbf{x}_{k})\rightarrow 0$. Therefore, from (\ref{betak}), for $k$ large
enough we can choose a sufficiently small $\delta $ such that
\begin{equation*}
\beta _{k}^{\prime }\leq (1+\delta )\beta _{k}<\beta <1
\end{equation*}%
for $\min (\mathbf{x}_{k})$ sufficiently small. As a result, we have
\begin{equation*}
\tan \varphi _{k+1}\leq \beta \tan \varphi _{k}
\end{equation*}%
for $k\geq N$ with $N$ large enough and $\min (\mathbf{x}_{k})$ sufficiently
small. This means that $\tan \varphi _{k}\rightarrow 0$, i.e., $\sin \varphi
_{k}\rightarrow \mathbf{0}$. From (iii) of Lemma~\ref{opposite}, $\sin
\varphi _{k}$ is uniformly bounded below by a positive constant. 
So $\sin \varphi_{k}\rightarrow \mathbf{0}$ and $\sin \varphi _{k} \geq m$,
a contradiction. Therefore the initial assumption \textquotedblleft $%
\lim_{k\rightarrow \infty }\overline{\lambda }_{k}=\alpha >\rho (A^{-1})$%
\textquotedblright must be false.
\end{proof}

\subsection{Convergence Rates}

Theorem~\ref{main} has proved the global convergence of INI, but the results are only qualitative and do not tell us anything about how fast the INI method converges. In this subsection, we will show the convergence rate of INI with different relaxation factors $\gamma _{k}$. More precisely, we prove that INI converges at least linearly with an asymptotic convergence factor
bounded by $\frac{2\gamma }{1+\gamma }$ for $0\leq \gamma _{k}\leq \gamma <1$ and superlinearly for decreasing $%
\gamma _{k}=\frac{\overline{\lambda }_{k-1}-\overline{\lambda }_{k}}{%
\overline{\lambda }_{k-1}}$, respectively.

From (\ref{eq:updateeig}) we have
\begin{align}
\varepsilon _{k+1}& =\varepsilon _{k}\left( 1-\left( 1-\gamma_k\right) \min
\left( \frac{\mathbf{x}_{k}}{\varepsilon _{k}\mathbf{y}_{k+1}}\right) \right)
\notag \\
& =:\varepsilon _{k}\rho _{k}.  \label{eq:ratioerr}
\end{align}%
Since $\overline{\lambda }_{k}-\overline{\lambda }_{k+1}<\overline{\lambda }%
_{k}-\rho (A^{-1})$, from (\ref{eq:ratioerr}) and (\ref{eq:updateeig}), we
have
\begin{equation*}
\rho _{k}=1-\left( 1-\gamma_k\right) \min \left( \frac{\mathbf{x}_{k}}{%
\varepsilon _{k}\mathbf{y}_{k+1}}\right) =\left( 1-\frac{\overline{\lambda }%
_{k}-\overline{\lambda }_{k+1}}{\overline{\lambda }_{k}-\rho (A^{-1})}%
\right) <1.
\end{equation*}

\begin{theorem}
\label{linearconv}For INI, we have $\rho _{k}\leq \frac{2\gamma _{k}}{%
1+\gamma _{k}}<1.$ Moreover, if $\gamma _{k}\leq \gamma <1,$ then $\underset{%
k\rightarrow \infty }{\lim }\rho _{k}\leq \frac{2\gamma }{1+\gamma }<1$,
i.e., the convergence of INI is at least globally linear. If $\underset{%
k\rightarrow \infty }{\lim }\gamma _{k}=0,$ then $\underset{k\rightarrow
\infty }{\lim }\rho _{k}=0,$ that is, the convergence of INI is globally
superlinear.
\end{theorem}

\begin{proof}
From (\ref{eq: xAf_k}) and (\ref{eq: yk+1bound}), we have%
\begin{equation*}
\min \left( \frac{\mathbf{x}_{k}}{\varepsilon _{k}\mathbf{y}_{k+1}}\right)
\geq \min \left( \frac{\mathbf{x}_{k}}{(1+\gamma _{k})\varepsilon
_{k}A_{k}^{-1}\mathbf{x}_{k}}\right) =\frac{1}{1+\gamma _{k}}\min \left(
\frac{\mathbf{x}_{k}}{\varepsilon _{k}A_{k}^{-1}\mathbf{x}_{k}}\right) .
\end{equation*}%
From (\ref{Bkinv}), we get
\begin{equation}
\varepsilon _{k}A_{k}^{-1}\mathbf{x}_{k}=\mathbf{x}\mathbf{x}^{T}\mathbf{x}%
_{k}-\mathbf{x}\mathbf{c}^{T}L_{k}^{-1}\mathbf{V}^{T}\mathbf{x}%
_{k}+\varepsilon _{k}VL_{k}^{-1}V^{T}\mathbf{x}_{k}.  \label{rela3}
\end{equation}%
From Theorem~\ref{main}, we know that $\underset{k\rightarrow \infty }{\lim }%
\mathbf{x}_{k}=\mathbf{x}$ and $\underset{k\rightarrow \infty }{\lim }%
\overline{\lambda }_{k}=\rho (A^{-1})$, from which it follows that $%
\varepsilon _{k}\rightarrow 0$ and $L_{k}^{-1}\rightarrow \left( \rho
(A^{-1})I-L\right) ^{-1}$. On the other hand, since $L_{k}^{-1}\rightarrow
(\rho (A^{-1})I-L)^{-1}$ and $\underset{k\rightarrow \infty }{\lim }V^{T}%
\mathbf{x}_{k}=V^{T}\mathbf{x}=0$, from (\ref{rela3}) we get
\begin{equation*}
\underset{k\rightarrow \infty }{\lim }\varepsilon _{k}A_{k}^{-1}\mathbf{x}%
_{k}=\mathbf{x}.
\end{equation*}%
Consequently, we obtain
\begin{align*}
\underset{k\rightarrow \infty }{\lim }\min \left( \frac{\mathbf{x}_{k}}{%
\varepsilon _{k}\mathbf{y}_{k+1}}\right) & \geq \frac{1}{1+\gamma _{k}}\min
\left( \underset{k\rightarrow \infty }{\lim }\frac{\mathbf{x}_{k}}{%
\varepsilon _{k}B_{k}^{-1}\mathbf{x}_{k}}\right) \\
& =\frac{1}{1+\gamma _{k}}\min \left( \frac{\mathbf{x}}{\mathbf{x}}\right) =%
\frac{1}{1+\gamma _{k}}>0,
\end{align*}%
leading to
\begin{equation}
\rho _{k}\leq 1-\frac{1-\gamma _{k}}{1+\gamma _{k}}=\frac{2\gamma _{k}}{%
1+\gamma _{k}}<1.  \label{convfactor}
\end{equation}
\end{proof}

It can be seen from (\ref{convfactor}) that if $\gamma _{k}$ is small then INI must ultimately converge quickly.
Although Theorem~\ref{linearconv} has established the superlinear convergence of INI, it does not reveal the convergence order. Our next concern is to derive the precise convergence order of INI. This is more informative and instructive because it lets us understand how fast INI converges.

\begin{theorem}
\label{quadconvini} If the inner tolerance $\xi _{k}$ in INI satisfies
condition (\ref{eq:strategy}) with the relaxation factors
\begin{equation}
\gamma _{k}=\frac{\overline{\lambda }_{k-1}-\overline{\lambda }_{k}}{%
\overline{\lambda }_{k-1}},  \label{innertol}
\end{equation}%
then INI converges quadratically (asymptotically) in the form of
\begin{equation}
\overline{\varepsilon }_{k}\leq 2\overline{\varepsilon }_{k-1}^{2}
\label{convorder}
\end{equation}%
for $k$ large enough, where the relative error $\overline{\varepsilon }%
_{k+1}=\varepsilon _{k}/\rho (A^{-1})$.
\end{theorem}

\begin{proof}
Since $\overline{\lambda }_{k-1}>\overline{\lambda }_{k}>\rho (A^{-1}),$ we
have%
\begin{equation}
\gamma _{k}=\frac{\overline{\lambda }_{k-1}-\overline{\lambda }_{k}}{%
\overline{\lambda }_{k-1}}\leq \frac{\overline{\lambda }_{k-1}-\rho (A^{-1})%
}{\rho (A^{-1})}=\frac{\varepsilon _{k-1}}{\rho (A^{-1})}.  \label{eq: gamma}
\end{equation}%
From (\ref{eq:ratioerr}), (\ref{convfactor}) and (\ref{eq: gamma}), we have%
\begin{eqnarray*}
\varepsilon _{k} &=&\varepsilon _{k-1}\rho _{k-1}\leq \varepsilon _{k-1}%
\frac{2\gamma _{k}}{1+\gamma _{k}}=\varepsilon _{k-1}\frac{2}{1+\frac{1}{%
\gamma _{k}}} \\
&\leq &\varepsilon _{k-1}\frac{2}{1+\frac{\rho (A^{-1})}{\varepsilon _{k-1}}}%
=\varepsilon _{k-1}^{2}\frac{2}{\varepsilon _{k-1}+\rho (A^{-1})} \\
&\leq &\frac{2}{\rho (A^{-1})}\varepsilon _{k-1}^{2}.
\end{eqnarray*}%
Dividing both sides of the above inequality by $\rho (A^{-1})$, we get
\begin{equation*}
\overline{\varepsilon }_{k}=\frac{\varepsilon _{k}}{\rho (A^{-1})}\leq \frac{%
2}{\rho (A^{-1})^{2}}\varepsilon _{k-1}^{2}=2\overline{\varepsilon }%
_{k-1}^{2}.
\end{equation*}
\end{proof}

\section{The modified inexact Noda iteration}

In this section, we propose a modified Noda iteration (MNI) for a non-negative matrix, and show that MNI and NI are equivalent. Thus, by combining INI (Algorithm \ref{alg:ini}) with MNI we can propose a modified inexact Noda iteration for a monotone matrix

\subsection{The modified Noda iteration}
When $\widehat{\lambda }_{k}I-B$ tends to a singular matrix, the Noda
iteration requires us to solve a possibly ill-conditioned linear system (\ref%
{eq:step1}). Hence, we propose a rank one update technique for the ill-conditioned
linear system (\ref{eq:step1}), i.e.,%
\begin{equation}
\left(
\begin{array}{cc}
B-\widehat{\lambda }_{k}I & -\mathbf{x}_{k} \\
-\mathbf{x}_{k}^{T} & 0%
\end{array}%
\right) \left(
\begin{array}{c}
\Delta \mathbf{y}_{k} \\
\delta _{k}%
\end{array}%
\right) =\left[
\begin{array}{c}
(\widehat{\lambda }_{k}I-B)\,\mathbf{x}_{k} \\
0%
\end{array}%
\right] ,  \label{eq: NewtonNI}
\end{equation}%
where $\Delta \mathbf{y}_{k}=$ $\mathbf{x}_{k+1}-\mathbf{x}_{k}.$ Let $%
\mathbf{r}_{k}=(\widehat{\lambda }_{k}I-B)\,\mathbf{x}_{k}.$
In general, the linear system (\ref{eq: NewtonNI}) is a well-conditioned linear system, unless $B$ has the Jordan form corresponding to the largest eigenvalue, which contradicts the Perron--Frobenius theorem.

From (\ref{eq: NewtonNI}),
\begin{eqnarray*}
0 &=&\left( B-\widehat{\lambda }_{k}I\right) \left( \mathbf{x}_{k+1}-\mathbf{%
x}_{k}\right) -\delta _{k}\mathbf{x}_{k}-\mathbf{r}_{k} \\
&=&\left( B-\widehat{\lambda }_{k}I\right) \mathbf{x}_{k+1}-\left( B-%
\widehat{\lambda }_{k}I\right) \mathbf{x}_{k}-\delta _{k}\mathbf{x}_{k}-%
\mathbf{r}_{k} \\
&=&\left( B-\widehat{\lambda }_{k}I\right) \mathbf{x}_{k+1}-\delta _{k}%
\mathbf{x}_{k}.
\end{eqnarray*}%
Hence, we have the following linear system%
\begin{equation*}
\left( \widehat{\lambda }_{k}I-B\right) \left( \frac{\mathbf{x}_{k+1}}{%
-\delta _{k}}\right) =\mathbf{x}_{k},
\end{equation*}%
or%
\begin{equation*}
\left[ \widehat{\lambda }_{k}I-B\right] \mathbf{y}_{k+1}=\mathbf{x}_{k},
\end{equation*}%
with $\mathbf{y}_{k}=\frac{-\mathbf{1}}{\delta _{k}}\mathbf{x}_{k+1}.$ Thus,
from (\ref{eq:step1}) and (\ref{eq: NewtonNI}), we have the new iterative vector
\begin{equation}
\mathbf{x}_{k+1}=\frac{\mathbf{y}_{k+1}}{\left\Vert \mathbf{y}%
_{k+1}\right\Vert }=\frac{\mathbf{x}_{k}+\Delta \mathbf{y}_{k}}{\left\Vert
\mathbf{x}_{k}+\Delta \mathbf{y}_{k}\right\Vert }.  \label{eq: updatexk}
\end{equation}%
This means the Noda iteration and the modified Noda iteration are
mathematically equivalent. Based on (\ref{eq: NewtonNI}) and (\ref{eq:
updatexk}), we state our algorithm as follows.

\begin{algorithm}
\begin{enumerate}
  \item   Given $\widehat{\lambda}_{0}$, $\bx_0 > \zero$ with $\Vert \bx_0\Vert =1$ and ${\sf tol}>0$.
  \item   {\bf for} $k =0,1,2,\dots$
  \item   \quad {\bf if} $\Vert B\bx_{k+1}-\widehat{\lambda}_{k}\bx_{k+1}\Vert > {\sf \sqrt{tol}}$
  \item   \quad\quad Solve $(\widehat{\lambda}_{k}I-B)\,\mathbf{y}_{k+1}=\mathbf{x}_{k}$.
  \item   \quad\quad Normalize the vector $\bx_{k+1}= \by_{k+1}/\Vert \by_{k+1}\Vert$.
  \item   \quad {\bf else if}
  \item   \quad\quad Solve $\left(
\begin{array}{cc}
B-\widehat{\lambda}_{k} & -\mathbf{x}_k \\
-\mathbf{x}_k^{T} & 0%
\end{array}%
\right)\left(
\begin{array}{c}
\Delta\mathbf{y}_{k} \\
\delta _{k}%
\end{array}%
\right)  =\left[
\begin{array}{c}
\widehat{\lambda}_{k}\mathbf{x}_{k}-B\mathbf{x}_{k} \\
0%
\end{array}%
\right].$
  \item   \quad\quad Normalize the vector $\bx_{k+1}= (\bx_{k}+\Delta \by_{k})/\Vert \bx_{k}+\Delta \by_{k}\Vert$.
  \item   \quad {\bf end}
  \item   \quad Compute $\widehat{\lambda}_{k+1}= \max \left(\frac{B\bx_{k+1}}{\bx_{k+1}}\right)$.
  \item   {\bf until} convergence: $ \Vert B\bx_{k+1}-\widehat{\lambda}_{k}\bx_{k+1}\Vert <{\sf tol}$.
\end{enumerate}
\caption{Modified Noda Iteration (MNI)}
\label{alg:MNI}
\end{algorithm}

Note that the sequence $\{\widehat{\lambda }_{k}I-B\}$ tends to a singular
matrix, meaning that $\{\widehat{\lambda }_{k}\}$ converges to an eigenvalue
of $B$, and (\ref{eq:step1}) becomes an ill-conditioned linear system. Based on
practical experiments, we propose taking $\Vert B\bx_{k+1}-\widehat{\lambda}_{k}\bx_{k+1}\Vert \leq \sqrt{\mathrm{tol}}$ and
switching from (\ref{eq:step1}) to (\ref{eq: NewtonNI}).

\subsection{The modified inexact Noda iteration}
For a monotone matrix $A$, we replaced $B$ by $ A^{-1}$ in (\ref{eq: NewtonNI}).
The linear system (\ref{eq: NewtonNI}) can be rewritten as
\begin{equation}
\left(
\begin{array}{cc}
I-\widehat{\lambda}_{k}A & -A\mathbf{x}_k \\
-\mathbf{x}_k^{T} & 0%
\end{array}%
\right)\left(
\begin{array}{c}
\Delta\mathbf{y}_{k} \\
\delta _{k}%
\end{array}%
\right)  =\left[
\begin{array}{c}
\widehat{\lambda}_{k}A\mathbf{x}_{k}-\mathbf{x}_{k} \\
0%
\end{array}%
\right]. \label{eq: NewtonNI2}
\end{equation}
Based on MNI, by combining INI (Algorithm \ref{alg:ini}) with equation (\ref{eq: NewtonNI2}), we can propose a modified inexact Noda iteration for a monotone matrix, which is described as Algorithm~\ref{alg:MINI_M}.
\begin{algorithm}
\begin{enumerate}
  \item   Given $\widehat{\lambda}_{0}$, $\bx_0 > \zero$ with $\Vert \bx_0\Vert =1$ and ${\sf tol}>0$.
  \item   {\bf for} $k =0,1,2,\dots$
  \item   \quad {\bf if} $= \Vert A\bx_{k+1}-\overline{\lambda }_{k}^{-1}\bx_{k+1}\Vert > {\sf \sqrt{tol}}$
  \item   \quad\quad Run INI for monotone matrix $A$ (Algorithm \ref{alg:ini}).
  \item   \quad {\bf else if}
  \item   \quad\quad Solve $\left(
\begin{array}{cc}
I-\overline{\lambda }_{k}A & -A\mathbf{x}_k \\
-\mathbf{x}_k^{T} & 0%
\end{array}%
\right)\left(
\begin{array}{c}
\Delta\mathbf{y}_{k} \\
\delta _{k}%
\end{array}%
\right)  =\left[
\begin{array}{c}
\overline{\lambda }_{k}A\mathbf{x}_{k}-\mathbf{x}_{k} \\
0%
\end{array}%
\right]$ exactly.
  \item   \quad\quad Normalize the vector $\bx_{k+1}= (\bx_{k}+\Delta \by_{k})/\Vert \bx_{k}+\Delta \by_{k}\Vert$.
  \item   \quad\quad Compute $\overline{\lambda }_{k+1}$  that satisfies condition (\ref{eq:updateeig}).
  \item   \quad {\bf end}
  \item   {\bf until} convergence: $\Vert A\bx_{k+1}-\overline{\lambda }_{k}^{-1}\bx_{k+1}\Vert <{\sf tol}$.
\end{enumerate}
\caption{Modified Inexact Noda Iteration (MINI)}
\label{alg:MINI_M}
\end{algorithm}

\section{Numerical experiments}

\label{sec:exp}  In this section we present numerical experiments to support our theoretical results for INI, and to illustrate the effectiveness of the proposed MINI algorithms.
All numerical tests were performed on an Intel (R) Core (TM) i$7$ CPU $4770$%
@ $3.4$GHz with $16$ GB memory using Matlab R$2013a$ with the machine
precision $\epsilon =2.22\times 10^{-16}$ under the Microsoft Windows $7$ $%
64 $-bit.

$I_{\mathrm{outer}}$ denotes the number of outer iterations to achieve the
convergence, and $I_{\mathrm{inner}}$ denotes the total number of inner
iterations, which measures the overall efficiency of MNI and MINI. In view
of the above, we have the average number $I_{ave}=I_{inner}/I_{outer}$ at
each outer iteration for our test algorithms. In the tables,
\textquotedblleft Positivity\textquotedblright illustrates whether the
converged Perron vector preserves the strict positivity property. If
\textquotedblleft No\textquotedblright, then the percentage in the brace
indicates the proportion that the converged Perron vector has the positive
components. We also report the CPU time of each algorithm, which measures
the overall efficiency too.

\subsection{INI for computing the smallest eigenvalue of a monotone matrix}

We present an example to illustrate the numerical behavior of NI, INI\_1 and
INI\_2 for monotone matrices. The approximate solution $\mathbf{y}_{k+1}$
of (\ref{eq:inexactsys2}) satisfies%
\begin{equation*}
(\overline{\lambda}_kA-I) \, \mathbf{y}_{k+1}= A\mathbf{x}_k+\mathbf{f}_k
\end{equation*}
by requiring the following inner tolerances:

\begin{itemize}
\item for NI: $\Vert \mathbf{f}_k\Vert \le 10^{-14}$;

\item for INI\_1: $\Vert \mathbf{f}_k\Vert \le \gamma_k \mathrm{sep}%
(0,A)\min(\mathbf{x}_k)$ with some $0< \gamma_k< 1$;

\item for INI\_2: $\Vert \mathbf{f}_k\Vert \le \frac{\overline{\lambda}%
_{k-1}-\overline{\lambda}_k}{\overline{\lambda}_{k-1}}\mathrm{sep}(0,A)\min(%
\mathbf{x}_k)$ for $k \ge 1$ and $\overline{\lambda}_0 > \rho(A^{-1})$.
\end{itemize}
We use the minimal residual method to solve the inner linear systems.
For the implementations, we use the standard Matlab function \textsf{minres}.
The outer iteration starts with the normalized vector of $\left(1,\dots,1\right)^T$,
and the stopping criterion for outer iterations is
\begin{equation*}
\frac{\Vert A\mathbf{x}_k - \overline{\lambda}_k^{-1} \mathbf{x}_k\Vert} {%
(\|A\|_1\|A\|_{\infty})^{1/2}} \le 10^{-10},
\end{equation*}
where $\|\cdot\|_1$ and $\|\cdot\|_{\infty}$ are the one norm and the
infinity norm of a matrix, respectively.

Condition (\ref{eq:strategy}) ensures that the eigenvector in Lemma \ref{monotone2} does indeed preserve the strict positivity property. However, the formula in (\ref{eq:strategy}) is not applicable in practice, because it uses $\mathrm{sep}(0,A)$, which is unknown at the time it needs to be computed . Therefore, for practical implementations, we suggest a relaxation strategy to replace
$\mathrm{sep}(0,A)$ by $\overline{\lambda}_k^{-1}$. The quantity $\overline{%
\lambda}_k^{-1}$ is related to the lower bound of the smallest eigenvalue of
$A$, i.e., $\sigma_{min}(A)\geq \overline{\lambda}_k^{-1}$.
For all examples, the stopping criterion for the inner iteration is set at
\begin{equation*}
\Vert \mathbf{f}_{k}\Vert \leq \max \{\gamma _{k}\min (\mathbf{x}_{k})/%
\overline{\lambda }_{k},10^{-13}\}\text{ for INI\_1}
\end{equation*}%
and%
\begin{equation*}
\Vert \mathbf{f}_{k}\Vert \leq \max \{\frac{\overline{\lambda }_{k-1}-%
\overline{\lambda }_{k}}{\overline{\lambda }_{k-1}\overline{\lambda }_{k}}%
\min (\mathbf{x}_{k}),10^{-13}\}\text{ for INI\_2.}
\end{equation*}

\begin{example}
\label{exp:FEM}We consider the finite-element discretization of the boundary
value problem in \cite[Example 4.2.4]{Axels90}
\begin{eqnarray*}
-u_{xx}-u_{yy} &=&g(x,y)\text{ \ \ in }\Omega =[0,a]\times \lbrack 0,b], \\
a,b &>&0,\text{ \ \ \ \ \ \ \ \ }u=f(x,y)\text{ on }\partial \Omega ,
\end{eqnarray*}%
using piecewise quadratic basis functions on the uniform mesh of $p\times m$
isosceles right-angled triangles. This is a matrix of order $%
n=(2p-1)(2m-1)=127,041$ with $p=400$ and $m=80$.
\end{example}

For Example 1, we see that, for this monotone matrix eigenproblem, INI\_1, with two different $\gamma_k=0.5$ and $0.8$ exhibits distinct convergence behaviors and uses $51$ and $18$ outer iterations to achieve the desired accuracy, respectively. As Figure~\ref{fig:FEM} indicates, NI and INI\_2 typically converge superlinearly, and INI\_1 with $\gamma_k=0.5, 0.8$ typically converge linearly. This confirms our theory and demonstrates that the results of our theorem can be realistic and pronounced.

\begin{figure}
\centering
\epsfig{file=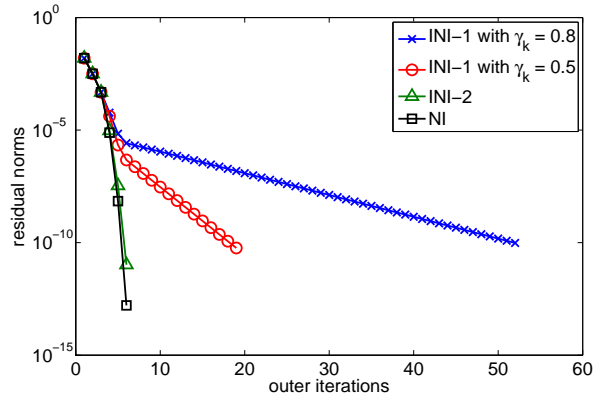,width=.6\textwidth}
\caption{The outer residual norms versus outer iterations in Examples
\protect\ref{exp:FEM}.}
\label{fig:FEM}
\end{figure}

\begin{table}
\caption{The total outer and inner iterations in Example \protect\ref%
{exp:FEM}}
\label{table1}
\begin{tabular}{l|rrrrl}
\hline
Method & $I_{\mathrm{outer}}$ & $I_{\mathrm{inner}}$ & $I_{\mathrm{ave}}$ &
CPU time & Positivity \\ \hline
INI\_1 with $\gamma =0.8$ & 51 & 19622 & 384 & 76 & Yes \\
INI\_1 with $\gamma =0.5$ & 18 & 11233 & 624 & 38 & Yes \\
NI & 5 & 3621 & 724 & 25 & Yes \\
INI\_2 & 5 & 3591 & 718 & 19 & Yes \\ \hline
\end{tabular}%

\end{table}

We observe from Table~\ref{table1} that all the converged eigenvectors are positive, and that INI\_2 improves the overall efficiency of NI. As we see, the INI\_1 algorithm converges linearly and slowly. To be precise, INI\_1 needs between twice and three times the CPU time of INI\_2, but $I_{\mathrm{ave}}$ for INI\_1 is only half $I_{\mathrm{ave}}$ of INI\_2. There are two reasons for this. First, since the approximate eigenvalues are obtained from the relation (\ref{eq:updateeig}), then the parameter $\gamma_k$ will lead to a difference in the convergence rates, as is seen from Figure~\ref{fig:FEM}. Second, from (\ref{eq:strategy}), INI\_2 solves the inner linear systems more and more accurately as $k$ increases . In contrast, the inner tolerance used by INI\_1 is fixed except for the factor $\min \left( \,\mathbf{x}_{k}\right)$, which also makes the average number of the iterations of INI\_1 only about half of those for INI\_2.

\subsection{MINI for computing the smallest singular value of an $M$-matrix}

In the above section, INI\_2 was considerably better than NI and INI\_1 for overall efficiency. Therefore, in this subsection, we use MINI (INI\_2 combined with MNI) to find the smallest singular value and the associated eigenvector of an $M$-matrix, and confirm the effectiveness of MINI and the theory we presented in Sections~\ref{sec:type1} and \ref{sec:conv}. For MINI, the stopping criteria for inner and outer iterations are the same as those for monotone matrices. In the meantime , we compare MINI with the algorithms JDQR \cite{S02} and JDSVD \cite{Hoch01} and the Matlab function \textsf{svds}; none of these are positivity preserving for approximate eigenvectors.
We show that the MINI algorithm always reliably computes positive eigenvectors, while the other algorithms generally fail to do so.

Since JDQR and JDSVD use the absolute residual norms to decide the
convergence, then we set the stopping criteria ``TOL$=10^{-10}{%
(\|A\|_1\|A\|_{\infty})^{1/2}}$'' for outer iterations, and then we will get the same stopping criteria as used for MINI.
We set the parameters ``sigma=SM'' for
JDQR, ``opts.target=0'' for JDSVD, and the inner solver
``OPTIONS.LSolver=minres''. All the other options use defaults. We do not
use any preconditioning for inner linear systems. For \textsf{svds}, we set
the stopping criteria ``OPTS.tol$=10^{-10}{(\|A\|_1\|A\|_{\infty})^{1/2}}$,
and take the maximum and minimum subspace dimensions as $20$ and $2$ at each
restart, respectively.

Suppose that we want to compute the smallest singular value, and the
corresponding singular vector, of a real $n\times n$ M-matrix $M$. This
partial SVD can be computed by using equivalent eigenvalue decomposition,
that is, the augmented matrix
\begin{equation*}
A = \left[
\begin{array}{cc}
0 & M \\
M^{T} & 0%
\end{array}%
\right].
\end{equation*}
Obviously, such a matrix $A$ is no longer an $M$-matrix but will indeed be monotone.

\begin{example}
\label{exp:rgg_n_2_19_s0}We consider a symmetric $M$-matrix of the form $%
M=\sigma I-B$, where $B$ is the non-negative matrix \textsf{rgg\_n\_2\_19\_s0}
from the DIMACS10 test set \cite{DIMACS}. This matrix is a random geometric
graph with $2^{19}$ vertices. Each vertex is a random point in the unit
square and edges connect vertices whose Euclidean distance is below $0.55$ ($%
\log(n)/n$). This threshold is chosen in order to ensure that the graph is
almost connected. This matrix is a binary matrix with $n=2^{19}=524,288$ and
$6,539,532$ nonzero entries.
\end{example}

For this problem, MINI works very well and uses only six outer iterations to attain the desired accuracy. Furthermore, it is reliable and positivity preserving. In contrast, JDQR, JDSVD, and \textsf{svds} compute the desired eigenvalue, but the converged eigenvectors are not positive. More precisely, Table~\ref{table2} indicates that for these algorithms roughly 50\% of the components of each converged eigenvector are negative.

As far as overall efficiency is concerned, MINI is the most efficient in terms of $I_{\mathrm{inner}}$, $I_{\mathrm{outer}}$ and the CPU time. JDQR and \textsf{svds} require at least five times the CPU time of MINI; they are also more expensive than JDSVD in terms of the CPU time.
\begin{table}[ht]
\caption{The total outer and inner iterations in Example \protect\ref%
{exp:rgg_n_2_19_s0}}
\label{table2}
\begin{tabular}{l|rrrrl}
\hline
Method & $I_{\mathrm{outer}}$ & $I_{\mathrm{inner}}$ & $I_{ave}$ & CPU time
& Positivity \\ \hline
MINI & 6 & 331 & 55 & 30 & Yes \\
JDQR & 25 & 4068 & 162 & 243 & No (52\%) \\
JDSVD & 34 & 1432 & 42 & 58 & No (51\%) \\
\textsf{svds} & 140 & ----- & 140 & 144 & No (57\%) \\ \hline
\end{tabular}%

\end{table}

\section{Conclusions}

We have proposed an inexact Noda iteration method for computing the smallest eigenpair of a large irreducible monotone matrix, and have considered the convergence of the modified inexact Noda iteration with two relaxation factors. We have proved that the convergence of INI is globally linear and superlinear, with the asymptotic convergence factor bounded by $\frac{2\gamma _{k}}{1+\gamma
_{k}}$. More precisely, the modified inexact Noda iteration with inner tolerance $\xi
_{k}=\Vert \mathbf{f}_{k}\Vert \leq \gamma _{k}\mathrm{sep}(0,A)\min (%
\mathbf{x}_{k})$ converges at least linearly if the relaxation factors meet the condition $%
\gamma _{k}\leq \gamma <1$, and superlinearly if the relaxation
factors meet the condition $\gamma _{k}=\frac{\overline{\lambda }_{k-1}-\overline{\lambda }_{k}}{%
\overline{\lambda }_{k-1}}$, respectively. The results for INI clearly show how the accuracy
of the inner iterations affects the convergence of the outer iterations.

In the experiments, we also compared MINI with Jacobi--Davidson type methods (JDQR, JDSVD) and the implicitly restarted Arnoldi method (\textsf{svds}). The contribution of this paper is twofold. First, MINI always preserves the positivity of approximate eigenvectors, while the other three methods often fail to do so. Second, the proposed MINI algorithms have been shown to be practical and effective for large monotone matrix eigenvalue problems and $M$-matrix singular value problems.

\end{document}